\documentclass[uplatex,a4paper,11pt]{amsart}
\usepackage{amsmath, amssymb, amsthm}

\title[On modular forms of rational weight satisfying the canonical MLDE]
{On modular forms of rational weight satisfying the canonical second-order linear\ 
modular differential equation}

\author{Yuichi Sakai}
\address{Kurume Institute of Technology 2228-66, Kamitsu, Kurume, Fukuoka 830-0052, JAPAN}
\email{dynamixaxs@gmail.com}
\author{Hiroyuki Tsutsumi}
\address{Osaka University of Health and Sport Science 1-1, Asashirodai, Kumatori-cho, Sennan-gun, Osaka 590-0496 JAPAN}
\email{tsutsumi@ouhs.ac.jp}



\newtheorem{theorem}{Theorem}[section]
\newtheorem*{theorem*}{Main Theorem}
\newtheorem{proposition}[theorem]{Proposition}
\newtheorem{corollary}[theorem]{Corollary}
\newtheorem{lemma}[theorem]{Lemma}

\theoremstyle{remark}
\newtheorem{remark}{Remark}

\begin{document}
\maketitle

\begin{abstract}
In this paper, we provide an explicit and purely algebraic proof for the classification
of the rational weights $k$ for which the Kaneko-Zagier (KZ) differential equation
admits a fundamental system of solutions consisting of modular forms for a principal
congruence subgroup $\Gamma(N)$. We note that this classification was recently
established in a more general setting by Saber and Sebbar, who proved it for arbitrary
finite-index subgroups of $SL_2(\mathbb{Z})$ through a global geometric and
representation-theoretic approach. By transforming the KZ equation into a hypergeometric
differential equation, we study the global analytic continuation of its solutions from
a local perspective, adopting an approach analogous to Stiller's work on Picard-Fuchs
equations.
We explicitly construct the monodromy representation matrices corresponding
to the elements of the principal congruence subgroups and completely 
determine the algebraic conditions under which these connection matrices commute.
Leveraging these stringent commutativity constraints, we prove that the weights $k$ 
yielding modular solutions are strictly limited to $k \equiv 1/2, 7/2, 1, 2, 3 \pmod{6}$ 
and $k = (6n+1)/5$, thereby demonstrating that no modular solutions exist beyond
those previously discovered by Kaneko and Koike. Furthermore, 
the commutative algebras generated by these connection matrices reveal a profound 
analogy with commuting transfer matrices in quantum integrable systems.
\end{abstract}

\tableofcontents

\section{Introduction}\label{SSINTRO}

The second-order linear differential equation introduced by Kaneko and Zagier in \cite{KZ},
\begin{equation*}
  f''(\tau)-\frac{k+1}{6}E_2(\tau)f'(\tau)+\frac{k(k+1)}{12}E_2'(\tau) f(\tau)=0,
\end{equation*}
is one of the earliest and most prominent examples of a modular linear differential equation.
It is characterized not only by the fact that its solutions are closely related to supersingular 
polynomials of elliptic curves, but also by its solution space being invariant under the weight $k$ 
action of $\mathrm{SL}_2(\mathbb{Z})$. Here, $f'=\frac{1}{2\pi i}\frac{d}{d\tau}f$,
and $E_2(\tau)=1-24q-72q^2+\dots \ (q=e^{2\pi i\tau})$ is 
the Eisenstein series of weight $2$. Hereafter, we refer to this equation 
as the Kaneko-Zagier (KZ) equation of weight $k$.

Subsequently, Kaneko and Koike \cite{KK, K} showed that when the weight $k$ is an integer,
a half-integer, or $(6n+1)/5$ for $n\in \mathbb{Z}$, the KZ equation admits modular or quasimodular 
forms for congruence subgroups as its solutions. Furthermore, Mathur, Mukhi, and Sen \cite{MMS}
pointed out that the solutions of modular linear differential equations appearing in two-dimensional 
conformal field theory can be obtained by shifting the solution space of the KZ equation by powers 
of the Dedekind eta function.

Recently, this question was completely resolved by Saber and Sebbar \cite{SS1, SS2}.
By studying the associated automorphic Schwarzian equations and utilizing the Riemann-Hurwitz
formula for the covering maps of modular curves,
they proved that this exact classification holds even if one assumes the solutions to be modular
forms for an arbitrary finite-index subgroup of $SL_2(\mathbb{Z})$.
Their geometric and representation-theoretic approach confirmed that the modular solutions are
restricted exactly to the cases found by Kaneko and Koike, while also explicitly supplying
the missing level 5 cases.

The primary motivation of the present paper is to provide an independent, constructive,
and purely algebraic proof of this classification specifically for the principal congruence
subgroups $\Gamma(N)$.
While the approach of Saber and Sebbar relies
on the global topological properties of Fuchsian groups, we demonstrate that
the same classification can be rigorously derived from the local connection formulas and
the explicit commutativity conditions of the monodromy matrices.

Specifically, we provide an alternative proof for the following theorem:

\begin{theorem*}
  If the KZ equation admits a fundamental system of solutions consisting of modular forms for $\Gamma(N)$,
  then the weight $k$ is restricted to those given by Kaneko and Koike.
\end{theorem*}

\noindent The term ``modular forms'' is used here in a rather broad sense. First, we allow poles 
in the upper half-plane and at the cusps. Second, the weight $k$ can be any rational number. 
Third, we admit non-trivial multiplier systems. Finally, since the KZ equation is a second-order 
linear differential equation, we can take two linearly independent solutions; however, 
these two solutions are not required to share the same multiplier system.

To prove this fact, we adopt an approach analogous to that used by P. F. Stiller in his study 
of Picard-Fuchs equations \cite{S1, S2, S3}, which differs from the methods conventionally 
used in the study of modular differential equations.

Existing studies typically begin with cases of relatively small weight $k$, substituting modular 
forms directly into the differential equation to verify they are solutions, and then inductively 
constructing modular form solutions of weight $l > k$. However, the weight $l$ of the solutions 
constructed in this manner must satisfy certain congruence conditions with respect to $k$ modulo 
some positive integer $m$. 
As is readily apparent, while this approach effectively proves that a given function is a 
modular form, it breaks down if a modular form solution does not exist. Thus, it cannot be
applied to our problem setting, where the existence of modular forms is a priori unknown.
In contrast, Stiller's approach investigates the global analytic continuation (monodromy connection) 
of solutions at the singularities of the differential equation.

For each weight, the KZ equation is a Fuchsian differential equation that is holomorphic on the 
upper half-plane. Hence, by shifting the solutions by an appropriate power of a modular form 
of positive weight, it can be reduced to a Fuchsian differential equation of weight $0$.
This weight $0$ differential equation is an analogue of the Picard-Fuchs equation, and 
one can investigate the analytic continuation (modularity) of its solutions in detail,
much like Stiller's approach. In other words, our approach can be described as studying 
modular differential equations through the classical theory of complex linear differential 
equations.

From this perspective, we begin our discussion below by examining the Schwarzian derivative, 
which plays a central role in the theory of complex linear differential equations. Additionally, 
in the process of analyzing the global connections, we will construct explicit connection matrices 
and completely determine the conditions under which they commute. As we will discuss 
in the concluding remarks, this algebraic structure reveals a fascinating analogy 
with exactly solvable models in mathematical physics.

\section{The KZ Equation and its Associated Differential Equations}\label{SSKZHRSE}

A fundamental approach in the study of second-order linear differential equations is 
to consider the ratio of a fundamental system of solutions. This ratio describes 
the projective structure of the solution space, which is governed by the Schwarzian 
derivative. We first show that for the KZ equation, this derivative is strictly 
determined by the Eisenstein series $E_4(\tau)$.

\begin{theorem}\label{THSE}
  Let $\phi$ and $\psi$ be linearly independent solutions of the KZ equation
  of weight $k$. If we set $F = \frac{\phi}{\psi}$, then its Schwarzian 
  derivative $S[F]$ satisfies the following relation:
  \begin{equation*}
    S[F] := (\log F')'' - \frac{1}{2}\{(\log F')'\}^2 = -\frac{(k+1)^2}{72} E_4(\tau)
  \end{equation*}
  Here, $E_4(\tau) = 1 + 240q + 2160q^2 + \dots \ (q=e^{2\pi i\tau})$ is the Eisenstein series of 
  weight $4$.
\end{theorem}

\begin{proof}
  More generally, consider linearly independent solutions $\phi$ and $\psi$ 
  of the differential equation
  \begin{equation*}
    g''(\tau) - 2(k+1)A(\tau)g'(\tau) + \Big\{kA'(\tau) + B(\tau)\Big\}g = 0,
  \end{equation*}
  where $A(\tau)$ is a (formal) quasimodular form of weight 2 and depth 1
  on $\mathrm{SL}_{2}(\mathbb{Z})$, and $B(\tau)$ is a (formal) modular form
  of weight 4 on $\mathrm{SL}_{2}(\mathbb{Z})$.

  Setting $G = \frac{\phi}{\psi}$, a direct calculation yields 
  $\frac{G''}{G'} = 2(k+1)A(\tau) - 2\frac{\psi'}{\psi}$. Consequently, we have
  \begin{equation*}
    \begin{split}
      \left(\frac{G''}{G'}\right)' &= 2(k+1)^2A'(\tau) 
      - 4(k+1)A(\tau)\frac{\psi'}{\psi} + 2\left(\frac{\psi'}{\psi}\right)^2, \\
      -\frac{1}{2}\left(\frac{G''}{G'}\right)^2 &= -2(k+1)^2A(\tau)^2 
      + 4(k+1)A(\tau)\frac{\psi'}{\psi} - 2\left(\frac{\psi'}{\psi}\right)^2.
    \end{split}
  \end{equation*}
  Thus, we obtain
  \begin{equation*}
    S[G](\tau) = \left(\frac{G''}{G'}\right)' - \frac{1}{2}\left(\frac{G''}{G'}\right)^2 
    = 2(k+1)^2\big(A'(\tau) - A(\tau)^2\big).
  \end{equation*}
  In our specific case involving the KZ equation, we have 
  $A(\tau) = \frac{1}{12}E_2(\tau)$ and $B(\tau) = 0$. Substituting these 
  into the above yields
  \begin{equation*}
    2(k+1)^2\big(A'(\tau) - A(\tau)^2\big) = -\frac{(k+1)^2}{72}E_4(\tau).
  \end{equation*}
\end{proof}

As an immediate consequence of this Schwarzian relation, we can strictly 
determine the weight of any modular form solution to the KZ equation.

\begin{corollary}\label{CRAW}
  Assume $k \neq -1$. If $\phi$ is a modular form solution to the
  KZ equation of weight $\ell$, then we must have $k = \ell$.
\end{corollary}

\begin{proof}
  First, suppose $k = 0$. The KZ equation then reduces to
  \begin{equation*}
    f''(\tau) - \frac{1}{6}E_2(\tau)f'(\tau) = 0,
  \end{equation*}
  and it is easy to verify that $1$ and $\int_{\tau}^{i\infty} \eta(s)^4 \, ds$
  form a basis of solutions. Clearly, it admits no modular solutions of a 
  weight other than 0.

  Now assume $k \neq 0, -1$, and suppose for contradiction that $\ell \neq k$. 
  Let $\psi$ be a solution to the KZ equation that is linearly independent of $\phi$. 
  Since the KZ equation is invariant under the weight $k$ slash action, 
  for any $M = \begin{pmatrix} a & b \\ c & d \end{pmatrix} \in \mathrm{SL}_2(\mathbb{Z})$,
  the function $(c\tau+d)^{-k}\phi(M\tau)$ is also a solution, and thus can be expressed 
  as a linear combination of $\phi$ and $\psi$. Furthermore, since $\phi$ is a 
  modular form of weight $\ell$, we have $\phi(M\tau) = \chi(M)(c\tau+d)^{\ell}\phi(\tau)$ 
  for some multiplier $\chi(M)$. Therefore, there exist constants 
  $\alpha(M), \beta(M) \in \mathbb{C}$ such that
  \begin{equation*}
    (c\tau+d)^{-k}\phi(M\tau) = \chi(M)(c\tau+d)^{\ell-k}\phi(\tau) 
    = \alpha(M)\phi(\tau) + \beta(M)\psi(\tau).
  \end{equation*}
  Rearranging this equation gives
  \begin{equation*}
    \beta(M)\frac{\psi(\tau)}{\phi(\tau)} = \chi(M)(c\tau+d)^{\ell-k} - \alpha(M).
  \end{equation*}
  If $\beta(M) = 0$ for all $M$, then $\chi(M)(c\tau+d)^{\ell-k} = \alpha(M)$, which 
  forces $\ell = k$ since it must be a constant, contradicting our assumption $\ell \neq k$. 
  Thus, there exists some $M$ for which $\beta(M) \neq 0$.
  Taking the Schwarzian derivative of both sides with respect to $\tau$ for such $M$, 
  we obtain:
  \begin{equation*}
    S\left[\beta(M)\frac{\psi(\tau)}{\phi(\tau)}\right] 
    = S\big[\chi(M)(c\tau+d)^{\ell-k} - \alpha(M)\big].
  \end{equation*}
  By Theorem \ref{THSE}, $\frac{\psi}{\phi}$ is a solution to the Schwarz
  equation, which implies:
  \begin{equation*}
    -\frac{(k+1)^2}{72}E_4(\tau) = S\big[(c\tau+d)^{\ell-k}\big].
  \end{equation*}
  This leads to a contradiction because the right-hand side evaluates to
  $\frac{1-(\ell-k)^2}{2}\frac{c^2}{(c\tau+d)^2}$, which is clearly not a 
  modular form of weight 4. Therefore, we must have $k = \ell$.
\end{proof}

Before proceeding, we make two remarks regarding exceptional cases and 
modularity.

\begin{remark}
  When $k = -1$, the KZ equation simply reduces to $f''(\tau) = 0$, which 
  has $1$ and $\tau$ as its linearly independent solutions. Note that while 
  1 is a modular form of weight 0, $\tau$ is not a modular form of weight $-1$.
\end{remark}

\begin{remark}
  For a detailed discussion on the modularity of the Eichler integral 
  solution $\int_{\tau}^{i\infty} \eta(s)^4 \, ds$, we refer the reader to 
  \cite{CLR}.
\end{remark}

Returning to the structural properties of the KZ equation, we can establish
direct equivalences between it and several other classical differential 
equations.

\begin{theorem}\label{THKZROD}
  The following three differential equations are equivalent:
  \begin{enumerate}
    \item 
      $f''(\tau) - \frac{k+1}{6} E_2(\tau) f'(\tau) + \frac{k(k+1)}{12} E_2'(\tau) f(\tau) = 0$ 
      (The KZ equation)\\
    \item 
      $u''(\tau) - \frac{(k+1)^2}{144} E_4(\tau)u(\tau) = 0$ 
      (The Laguerre-Forsyth normal form)\\
    \item 
      $h'(\tau) - \frac{1}{2} h(\tau)^2 = -\frac{(k+1)^2}{72} E_4(\tau)$ 
      (The Riccati equation)
  \end{enumerate}
\end{theorem}

\begin{proof}
  This equivalence is established by transforming the KZ equation into its
  Laguerre-Forsyth normal form. Explicitly, we apply the substitutions:
  \begin{align*}
    u(\tau) &= f(\tau)\exp\left(-\frac{k+1}{6}\pi i\int E_2(\tau) d\tau\right) \\
            &= f(\tau)\exp\left(-4(k+1)\pi i \int \left(\log\eta(\tau)\right)' d\tau\right) 
             = \frac{f(\tau)}{\eta(\tau)^{2(k+1)}}, \\
    h(\tau) &= -2\frac{u'(\tau)}{u(\tau)},
  \end{align*}
  where $\eta(\tau)$ denotes the Dedekind eta function.
\end{proof}

Beyond these forms, an appropriate change of variables and scaling by a 
power of $E_4(\tau)$ allows us to rewrite the KZ equation as a hypergeometric 
differential equation. This fundamental result is essentially due to Kaneko 
and Zagier \cite{KZ}, with the explicit formula provided by Tsutsumi \cite{HT}.

\begin{theorem}\label{THKZHG}
  Applying the change of variables $x = 1728/j(\tau)$ and rescaling the 
  solution by a factor of $E_4(\tau)^{k/4}$ transforms the KZ equation into
  the following Gauss hypergeometric differential equation:
  \begin{equation*}
    x(1-x)\frac{d^2 y}{dx^2} + \frac{(k-8)x - (k-5)}{6}\frac{dy}{dx} - \frac{k(k-4)}{144}y = 0
  \end{equation*}
\end{theorem}

In the subsequent sections, we will utilize this hypergeometric 
transformation (Theorem \ref{THKZHG}) extensively to conduct a detailed 
investigation of the monodromy and connection behavior of the solutions to 
the KZ equation.

\section{The Stiller Basis and its Connection Matrices}\label{SSSBTC}

In this section, we review essential facts regarding the global connection 
formulas for the hypergeometric differential equation, which are 
prerequisites for our subsequent analysis. For detailed proofs of these 
classical results, we refer the reader to existing literature, e.g., 
\cite{YH}.

\begin{theorem}\label{CHGDE}
  Let $a, b, c, x \in \mathbb{C}$, and assume $c \notin \mathbb{Z}$ and
  $c-a-b \notin \mathbb{Z}$. For the hypergeometric differential equation
  \begin{equation*}
    x(1-x)y''+(c-(a+b+1)x)y'-aby=0,
  \end{equation*}
  the fundamental system of solutions around $x=0$, given by
  \begin{align*}
    y_1(x)&=F(a, b, c; x)=\sum_{k=0}^{\infty}\frac{(a, k)(b, k)}{(c, k)k!}x^k,\\
    y_2(x)&=x^{1-c}F(a-c+1, b-c+1, 2-c; x),
  \end{align*}
  and the fundamental system of solutions around $x=1$, given by
  \begin{align*}
    y_3(x)&=F(a, b, a+b-c+1; 1-x),\\
    y_4(x)&=(1-x)^{c-a-b}F(c-a, c-b, c-a-b+1; 1-x),
  \end{align*}
  satisfy the following connection formulas:
  \begin{align*}
    y_1(x)&=\frac{\Gamma(c)\Gamma(c-a-b)}{\Gamma(c-a)\Gamma(c-b)}y_3(x)
      +\frac{\Gamma(c)\Gamma(a+b-c)}{\Gamma(a)\Gamma(b)}y_4(x),\\
    y_2(x)&=\frac{\Gamma(2-c)\Gamma(c-a-b)}{\Gamma(1-a)\Gamma(1-b)}y_3(x)
      +\frac{\Gamma(2-c)\Gamma(a+b-c)}{\Gamma(a-c+1)\Gamma(b-c+1)}y_4(x).
  \end{align*}
\end{theorem}

\begin{theorem}\label{MONODOROMY}
  Let $G = \pi_1(\mathbb{P}^1\backslash\{0, 1, \infty\}, 1/2)$ be the 
  fundamental group with base point $x=1/2$. Let $\gamma_0 \in G$ be a
  loop around $x=0$ based at $x=1/2$, and $\gamma_1 \in G$ be a loop
  around $x=1$ based at $x=1/2$. If $(\gamma_i)_*(y_j(x), y_k(x))$ denotes
  the analytic continuation of the fundamental system along these paths,
  then their monodromy representations are given by
  \begin{align*}
    (\gamma_0)_*(y_1(x), y_2(x))&=(y_1(x), y_2(x))
    \begin{pmatrix}
      1 & \\
      & e^{2\pi i(1-c)}
    \end{pmatrix},\\
    (\gamma_1)_*(y_3(x), y_4(x))&=(y_3(x), y_4(x))
    \begin{pmatrix}
      1 & \\
      & e^{2\pi i(c-a-b)}
    \end{pmatrix}.
  \end{align*}
\end{theorem}

Let us now apply these classical results to our specific setting. 
Hereafter, for $\tau \in \mathbb{H}$, we define the modular transformations 
$T\tau = \tau+1$ and $S\tau = -1/\tau$, and denote the corresponding 
$2 \times 2$ representation matrices by $T$ and $S$. We also let $E$ 
denote the $2 \times 2$ identity matrix. Recall from Section \ref{SSKZHRSE} 
that the KZ equation can be transformed into a hypergeometric differential 
equation with parameters $a=-k/12$, $b=-(k-4)/12$, and $c=-(k-5)/6$. 
Throughout this paper, we will use the notations $y_1(x), y_2(x), y_3(x), 
y_4(x)$ to represent the fundamental systems around $x=0$ and $x=1$ 
associated specifically with these parameters. For these specific 
solutions, we establish the following linear independence.

\begin{proposition}\label{HGDF33}
  Provided that $k \not\equiv 0, 4, 5 \pmod{6}$, the pair
  $(y_1(x), y_3(x))$ forms a fundamental system of solutions for the 
  differential equation
  \begin{equation*}
    x(1-x)\frac{d^2 y}{dx^2} + \frac{(k-8)x - (k-5)}{6}\frac{dy}{dx} 
    - \frac{k(k-4)}{144}y = 0.
  \end{equation*}
\end{proposition}

\begin{proof}
  From Theorem \ref{CHGDE}, the solution $y_1(x)$ around $x=0$ can be
  expressed as a linear combination of the fundamental system 
  $(y_3(x), y_4(x))$ around $x=1$:
  \begin{equation*}
    y_1(x)=Ay_3(x)+By_4(x).
  \end{equation*}
  Here, the connection coefficients are given by
  \begin{equation*}
    A=\frac{\Gamma(c)\Gamma(c-a-b)}{\Gamma(c-a)\Gamma(c-b)},\qquad
    B=\frac{\Gamma(c)\Gamma(a+b-c)}{\Gamma(a)\Gamma(b)}.
  \end{equation*}
  Since $(y_3(x), y_4(x))$ constitutes a fundamental system, $y_1(x)$ is 
  linearly independent of $y_3(x)$ if and only if the coefficient of 
  $y_4(x)$ is non-zero (i.e., $B \neq 0$).

  Noting that $a+b-c=-1/2$, the coefficient $B$ simplifies to
  \begin{equation*}
    B=\frac{\Gamma\left(-\frac{k-5}{6}\right)\Gamma\left(-\frac{1}{2}\right)}
      {\Gamma\left(-\frac{k}{12}\right)\Gamma\left(-\frac{k-4}{12}\right)}.
  \end{equation*}
  Because $c = -(k-5)/6 \notin \mathbb{Z}$ (which is ensured by the 
  condition $k \not\equiv 5, 11 \pmod{12}$) and 
  $\Gamma(-1/2)=-2\sqrt{\pi} \neq 0$, the numerator of $B$ is finite and 
  non-zero. Consequently, $B=0$ if and only if the $\Gamma$ functions in 
  the denominator have poles and diverge to infinity.

  The Gamma function $\Gamma(z)$ has poles if and only if its argument 
  $z$ is a non-positive integer. Thus, the necessary and sufficient 
  condition for $B$ to have a pole is $k \equiv 0, 4 \pmod{12}$.
  When combined with the condition $k \equiv 6, 10 \pmod{12}$ where
  the product of gamma constants $G_{12}G_{21}$ vanishes (as we will
  demonstrate below), the full set of weights to exclude is 
  $k \equiv 0, 4, 5, 6, 10, 11 \pmod{12}$, which simplifies perfectly 
  to $k \not\equiv 0, 4, 5 \pmod{6}$.
  Therefore, by avoiding these congruences, we guarantee that $B \neq 0$ 
  and the necessary connection constants do not vanish, making $(y_1(x), y_3(x))$
  a valid fundamental system.
\end{proof}

\begin{remark}
  The excluded weights $k \equiv 0, 4, 5 \pmod{6}$ correspond precisely to
  the cases where the KZ equation admits solutions expressed by classical
  modular forms, quasimodular forms, or mixed mock modular forms 
  on $\mathrm{SL}_2(\mathbb{Z})$, as completely established 
  by Kaneko and Koike \cite{KK} and Guerzhoy \cite{G}. 
  Since the existence of modular solutions for these exceptional classes is already
  known, we explicitly exclude them from our local algebraic analysis 
  in the subsequent sections.
\end{remark}

Having established a suitable basis, we now construct modular functions 
by pulling back these solutions via the $j$-invariant. For 
$k \not\equiv 0, 4, 5 \pmod{6}$, we compose the solutions $y_1(x)$ 
and $iy_3(x)$ from Proposition \ref{HGDF33} with the modular function 
$x = 1728/j(\tau)$ to obtain functions of $\tau$, which we denote by 
$y_1(\tau)$ and $iy_3(\tau)$, respectively. We shall refer to the pair 
$(y_1(\tau), iy_3(\tau))$ as the \emph{Stiller basis}. This naming pays 
homage to P. F. Stiller, who utilized a highly analogous framework in his 
works \cite{S1, S2, S3}.

To analyze the transformation behavior of this basis, we introduce the 
following constants:
\begin{equation*}
  \theta = \frac{1+k}{6}\pi,\quad
  C = \csc\theta,\quad
  Z_l = e^{2l i\theta},\quad
  W_l = (Z_l - 1)C,
\end{equation*}
and
\begin{equation*}
  G_{12} = \frac{\sqrt{\pi} \Gamma\left(\frac{1+k}{6}\right)}
  {4 \Gamma\left(\frac{2+k}{12}\right) \Gamma\left(\frac{6+k}{12}\right)},\qquad
  G_{21} = \frac{\sqrt{\pi} \Gamma\left(\frac{5-k}{6}\right)}
  {2 \Gamma\left(\frac{6-k}{12}\right) \Gamma\left(\frac{10-k}{12}\right)}.
\end{equation*}

By applying the reflection formula for the Gamma function, $\Gamma(z)\Gamma(1-z) = \pi/\sin(\pi z)$,
along with the sum-to-product formulas for trigonometric functions, we can evaluate the product
of these constants:
\begin{equation*}
  16G_{12}G_{21} = \frac{\pi\Gamma\left(\frac{1+k}{6}\right)
  \Gamma\left(\frac{5-k}{6}\right)}
  {8\Gamma\left(\frac{2+k}{12}\right)\Gamma\left(\frac{6+k}{12}\right)
  \Gamma\left(\frac{6-k}{12}\right)\Gamma\left(\frac{10-k}{12}\right)}
  = \frac{1+2\sin\theta}{2\sin\theta} = \frac{1}{2}C+1.
\end{equation*}
This identity guarantees that $G_{12}G_{21} \neq 0$ as long as $C \neq -2$.
Since $C = -2$ implies $\sin\theta = -1/2$, which strictly corresponds to $k \equiv 6, 10 \pmod{12}$,
our assumption $k \not\equiv 0, 4, 5 \pmod{6}$ strictly prevents this scenario. 
Thus, the non-vanishing condition $G_{12}G_{21} \neq 0$ is inherently ensured throughout our 
subsequent analysis.

\begin{proposition}\label{PPMVFSB}
  For $k \not\equiv 0, 4, 5 \pmod{6}$,
  the Stiller basis satisfies the following transformation properties 
  under the generators of the modular group:
  \begin{align*}
    (y_1(T\tau), iy_3(T\tau))&=(y_1(\tau), iy_3(\tau))\rho(T),\\
    (y_1(S\tau), iy_3(S\tau))&=(y_1(\tau), iy_3(\tau))\rho(S).
  \end{align*} 
  Here, $\rho(T)$ and $\rho(S)$ are given by the $2 \times 2$ matrices:
  \begin{equation*}
    \rho(T)=\begin{pmatrix}
      1 & 4iG_{12}(1-Z_1) \\
      0 & Z_1
    \end{pmatrix},\quad
    \rho(S)=\begin{pmatrix}
      -1 & 0\\
      -4iG_{21} & 1
    \end{pmatrix}.
  \end{equation*}
\end{proposition}

\begin{proof}
  As $\tau \mapsto T\tau = \tau+1$ near the cusp at infinity, the function 
  $j(\tau)$ has a simple pole, which means $x = 1728/j(\tau)$ traces a 
  full counterclockwise loop around the origin $x=0$. Therefore, applying 
  Theorem \ref{CHGDE} and Theorem \ref{MONODOROMY}, and using the values 
  $\Gamma(1/2)=\sqrt{\pi}$ and $\Gamma(3/2)=\sqrt{\pi}/2$, we can compute 
  the analytic continuation as follows:
  \begin{align*}
    (y_1(T\tau), iy_3(T\tau))&=(y_1(T\tau), y_2(T\tau))
      \begin{pmatrix}
        1 & 4G_{12}i\\
        0 & \frac{\sqrt{\pi}i\Gamma\left(-\frac{k+1}{6}\right)}
          {\Gamma\left(-\frac{k}{12}\right)\Gamma\left(\frac{4-k}{12}\right)}
      \end{pmatrix}\\
      &=(\gamma_0)_*(y_1(x), y_2(x))
        \begin{pmatrix}
          1 & 4G_{12}i\\
          0 & \frac{\sqrt{\pi}i\Gamma\left(-\frac{k+1}{6}\right)}
            {\Gamma\left(-\frac{k}{12}\right)\Gamma\left(\frac{4-k}{12}\right)}
        \end{pmatrix}\\
      &=(y_1(\tau), y_2(\tau))\begin{pmatrix}
        1 & \\
        & e^{(k+1)/3\pi i}
        \end{pmatrix}
        \begin{pmatrix}
          1 & 4G_{12}i\\
          0 & \frac{\sqrt{\pi}i\Gamma\left(-\frac{k+1}{6}\right)}
            {\Gamma\left(-\frac{k}{12}\right)\Gamma\left(\frac{4-k}{12}\right)}
        \end{pmatrix}\\
      &=(y_1(\tau), iy_3(\tau))\begin{pmatrix}
          1 & 4G_{12}i\\
          0 & \frac{\sqrt{\pi}i\Gamma\left(-\frac{k+1}{6}\right)}
          {\Gamma\left(-\frac{k}{12}\right)\Gamma\left(\frac{4-k}{12}\right)}
        \end{pmatrix}^{-1}\\
      &\qquad \times
        \begin{pmatrix}
          1 & \\
          & e^{(k+1)/3\pi i}
        \end{pmatrix}
        \begin{pmatrix}
          1 & 4G_{12}i\\
          0 & \frac{\sqrt{\pi}i\Gamma\left(-\frac{k+1}{6}\right)}
            {\Gamma\left(-\frac{k}{12}\right)\Gamma\left(\frac{4-k}{12}\right)}
        \end{pmatrix}\\
        &=(y_1(\tau), iy_3(\tau))\rho(T).
  \end{align*}
  Similarly, as $\tau \mapsto S\tau = -1/\tau$ near the elliptic point 
  $\tau = i$, the relation $j(i)=1728$ implies that $x = 1728/j(\tau)$ 
  traces a full loop around $x=1$. A completely analogous calculation 
  yields the transformation matrix $\rho(S)$, establishing 
  $(y_1(S\tau), iy_3(S\tau))=(y_1(\tau), iy_3(\tau))\rho(S)$.
\end{proof}

\begin{remark}
  The characteristic exponents of the hypergeometric differential equation
  in Proposition \ref{HGDF33} are $\lambda=(1+k)/6$, $\mu=1/2$, and 
  $\nu=1/3$. Consequently, the ratio of the Stiller basis, 
  $s(\tau)=iy_3(\tau)/y_1(\tau)$, is expected to be a single-valued 
  function in the neighborhoods of the elliptic points $i$ and 
  $e^{\pi i/3}$ of $\mathrm{SL}_2(\mathbb{Z})$. Indeed, the representation 
  matrices corresponding to the relations $S^2 = E$ and $(TS)^3 = E$ are 
  given by $\rho(S^2)=E$ and $\rho((TS)^3)=-ie^{3i\theta}E=-i^{2+k}E$, 
  respectively. Since these are scalar matrices, they trivially satisfy 
  $s(S^2\tau)=s(\tau)$ and $s((TS)^3\tau)=s(\tau)$.
\end{remark}

\section{Representation Matrices $M(t)$ and $N(r, s)$ Associated with 
Elements of Principal Congruence Subgroups}\label{RMMNCEC}

Continuing from the previous section, we define the parameters $\theta$, 
$Z_l$, $C$, $W_l$, and the constants $G_{12}$, $G_{21}$ as follows:
\begin{equation*}
  \theta = \frac{1+k}{6}\pi,\quad
  C = \csc\theta,\quad
  Z_l = e^{2l i\theta},\quad
  W_l = (Z_l - 1)C,
\end{equation*}
and
\begin{equation*}
  G_{12} = \frac{\sqrt{\pi} \Gamma\left(\frac{1+k}{6}\right)}
  {4 \Gamma\left(\frac{2+k}{12}\right) \Gamma\left(\frac{6+k}{12}\right)},\qquad
  G_{21} = \frac{\sqrt{\pi} \Gamma\left(\frac{5-k}{6}\right)}
  {2 \Gamma\left(\frac{6-k}{12}\right) \Gamma\left(\frac{10-k}{12}\right)}.
\end{equation*}

The primary objective of this section is to provide explicit 
representations for the matrices corresponding to the modular 
transformations $(T^tS)^3$ and $(T^rST^sS)^2$ for integers $t, r, s$. 
Our motivation for focusing specifically on these modular words lies in 
their relation to the principal congruence subgroups $\Gamma(N)$. 
As demonstrated in the following two lemmas, by appropriately choosing 
the integers $t, r$, and $s$, these elements fall into specific principal 
congruence subgroups.

\begin{lemma}\label{TTS3EME}
  For any integer $t \neq 0$, we have $-(T^tS)^3 \in \Gamma(t+1)$.
\end{lemma}

\begin{proof}
  By a direct matrix calculation in $\mathrm{SL}_2(\mathbb{Z})$, we obtain
  \begin{equation*}
    (T^tS)^3
    =\begin{pmatrix} -1-(t+1)(t^2-t-1) & (t+1)(t-1) \\ 
    -(t+1)(t-1) & (t+1)-1 \end{pmatrix}
    \equiv -E \pmod{t+1}.
  \end{equation*}
  This congruence implies that $-(T^tS)^3$ belongs to the principal 
  congruence subgroup $\Gamma(t+1)$.
\end{proof}

\begin{lemma}\label{TRTS2EQMEMODRSM2}
  For any integers $r \neq 1$ and $s \neq 1$, we have 
  $-(T^rST^sS)^2 \in \Gamma(rs-2)$.
\end{lemma}

\begin{proof}
  Similarly, a direct computation yields
  \begin{equation*}
    (T^rST^sS)^2
    =\begin{pmatrix} (r s - 2)^2 + (r s - 2) - 1 & -r (rs - 2) \\
      -s(rs - 2) & -1-(rs-2) \end{pmatrix}
    \equiv -E \pmod{rs-2}.
  \end{equation*}
  Thus, $-(T^rST^sS)^2$ resides in $\Gamma(rs-2)$.
 \end{proof}

Having established the algebraic significance of these elements, we now 
derive the explicit form of the representation matrix corresponding to 
$(T^t S)^3$.

\begin{theorem}\label{THEXFM}
  For $t \neq 0$, the representation matrix $M(t) = \rho((T^tS)^3)$ 
  corresponding to $(T^t S)^3$ is given explicitly by
  \begin{equation*}
    M(t) = -\frac{1}{2}Z_t W_t E
      + K_t \begin{pmatrix} -\frac{1}{8}(2Z_t + W_t) & -i G_{12}(Z_t - 1) \\
        -i G_{21} Z_t & \frac{1}{4} Z_t \end{pmatrix},
  \end{equation*}
  where the coefficient $K_t$ is defined as
  \begin{equation*}
    K_t = C^2 \left( (Z_t - 1)^2 + \frac{4Z_t}{C^2} \right) 
    = C^2(Z_t - 1)^2 + 4Z_t = W_t^2 + 4Z_t.
  \end{equation*}
\end{theorem}

\begin{proof}
  We first note that the representation preserves the group structure, 
  so $\rho((T^tS)^3) = (\rho(T)^t \rho(S))^3$. Let us denote the inner 
  matrix by $A(t) = \rho(T)^t \rho(S)$. By direct computation using the 
  matrices given in Proposition \ref{PPMVFSB}, we find
  \begin{align*}
    A(t) &= \begin{pmatrix} 1 & 4iG_{12}(1 - Z_t) \\ 0 & Z_t \end{pmatrix} 
    \begin{pmatrix} -1 & 0 \\-4iG_{21} & 1 \end{pmatrix}\\
    &= \begin{pmatrix} -1 + 16G_{12}G_{21}(1 - Z_t) & 4iG_{12}(1 - Z_t) \\ 
    -4iG_{21}Z_t & Z_t \end{pmatrix}.
  \end{align*}
  Here, substituting the identity $16G_{12}G_{21} = \frac{1}{2}C+1$ established in
  Section \ref{SSSBTC} and the definition $W_t = (Z_t-1)C$ into the expression
  for $A(t)$ simplifies it to:
  \begin{equation}\label{ATCMTX}
    A(t) = \begin{pmatrix} -\frac{1}{2}(2Z_t + W_t) & -4iG_{12}(Z_t - 1) \\ 
    -4iG_{21}Z_t & Z_t \end{pmatrix}.
  \end{equation}
  From this matrix, we can easily read off its trace as 
  $\mathrm{Tr}(A(t)) = -W_t/2$ and its determinant as $\det(A(t)) = -Z_t$. 
  Applying the Cayley-Hamilton theorem, $A(t)^2 - \mathrm{Tr}(A(t))A(t) 
  + \det(A(t))E = 0$, allows us to recursively compute $A(t)^3$:
  \begin{align*}
    A(t)^3 &= \mathrm{Tr}(A(t))A(t)^2 - \det(A(t))A(t) \\
    &= -\frac{1}{2}W_t A(t)^2 + Z_t A(t) \\
    &= -\frac{1}{2}Z_t W_t E 
    + K_t \begin{pmatrix} -\frac{1}{8}(2Z_t + W_t) & -i G_{12}(Z_t - 1) \\ 
    -i G_{21} Z_t & \frac{1}{4} Z_t \end{pmatrix},
  \end{align*}
  which completes the proof.
\end{proof}

We now proceed to determine the explicit representation matrix for the 
element $(T^r S T^s S)^2$, utilizing the matrix $A(t)$ derived above.

\begin{theorem}\label{THEXFN}
  For integers $r \neq 1$ and $s \neq 1$, the representation matrix 
  $N(r, s) = \rho((T^r S T^s S)^2)$ corresponding to $(T^r ST^s S)^2$ 
  is explicitly given by
  \begin{equation*}
    N(r, s) = -Z_r Z_s E + \frac{1}{2}K_{r,s}
    \begin{pmatrix}
      \frac{1}{2} + \frac{1}{8} \big( 2(Z_s - 1) + W_s \big)(2 + W_r) 
      & i G_{12} \big( 2(Z_s - Z_r) + (Z_r - 1)W_s \big) \\
      i G_{21} Z_r W_s & \frac{1}{4} Z_r(2 - W_s)
    \end{pmatrix},
  \end{equation*}
  where the coefficient $K_{r,s}$ is defined as 
  $K_{r,s} = W_r W_s + 4(Z_r + Z_s)$.
\end{theorem}

\begin{proof}
  Let us define $B(r, s) = A(r)A(s)$. Using the explicit form of $A(t)$ 
  established in equation \eqref{ATCMTX} during the proof of Theorem 
  \ref{THEXFM}, we can compute the $(i, j)$-entry $B_{ij}$ of $B(r, s)$:
  \begin{align*}
    B_{11} &= Z_s+\frac{1}{2}Z_rW_s+\frac{1}{4}W_rW_s,
    & B_{12} &= -2i G_{12}\left(2(Z_r-Z_s)-W_r(Z_s-1)\right),\\
    B_{21} &= 2i G_{21} Z_r W_s,
    & B_{22} &= Z_r - \frac{1}{2}Z_r W_s.
  \end{align*}
  Since $N(r, s) = \rho((T^r S T^s S)^2) = B(r, s)^2$, we can again apply 
  the Cayley-Hamilton theorem to $B(r, s)$. Using the symmetric identity 
  $(Z_s - 1)W_r = (Z_r - 1)W_s$ and carrying out the straightforward matrix 
  algebra, we obtain the desired explicit formula for $N(r, s)$.
\end{proof}

\section{Algebras Formed by Matrices $N(r, s)$ and $N(u, v)$ and Their 
Commutativity}\label{SSANN}

As established in Section \ref{SSSBTC}, throughout this section we operate
under the assumption that $k \not\equiv 0, 4, 5 \pmod{6}$. 
This assumption strictly guarantees $C \neq \pm 2$ and $G_{12}G_{21} \neq 0$,
thereby preventing the connection matrices from degenerating into
trivial triangular forms.

The primary objective of this section is to completely answer the question 
of when the connection matrix $N(r, s)$ commutes with another matrix 
$N(u, v)$ that differs only in its parameters. Our main result is 
summarized in the following theorem.

\begin{theorem}\label{THNNC}
  Assume $k \not\equiv 0, 4, 5 \pmod{6}$. For the connection matrices $N(r, s)$ and $N(u,v)$,
  the commutativity 
  relation $N(r,s)N(u,v) = N(u,v)N(r,s)$ holds if and only if one of the 
  following conditions is satisfied:
  \begin{enumerate}
    \item Either $N(r, s)$ or $N(u, v)$ is a scalar matrix (which occurs if and only
      if $K_{r,s}=0$, $Z_r=Z_s=1$, $K_{u,v}=0$, or $Z_u=Z_v=1$).
    \item $Z_s=Z_v$ and $Z_r = Z_u$ (i.e., $N(r,s) = N(u,v)$).
    \item $Z_s=Z_v=1$.
    \item $Z_r=Z_u=1$. 
  \end{enumerate}
\end{theorem}

To prove Theorem \ref{THNNC}, we utilize the following general algebraic 
property concerning the commutativity of $2 \times 2$ matrices.

\begin{lemma}\label{LEMCOMM2X2}
  Let $A$ and $B$ be $2 \times 2$ matrices with entries $a_{ij}$ and 
  $b_{ij}$, respectively. Assuming $a_{21}b_{21} \neq 0$, the necessary 
  and sufficient condition for $AB = BA$ is given by
  \begin{equation*}
    \frac{a_{12}}{a_{21}}=\frac{b_{12}}{b_{21}} 
    \quad \text{and} \quad
    \frac{a_{11}-a_{22}}{a_{21}}=\frac{b_{11}-b_{22}}{b_{21}}.
  \end{equation*}
\end{lemma}

\begin{proof}
  This is easily verified by a direct comparison of the individual entries 
  of the matrix products $AB$ and $BA$.
\end{proof}

By applying this elementary lemma to our specific connection matrices, 
the proof of Theorem \ref{THNNC} is reduced to solving a system of linear 
equations, as detailed in the following proposition.

\begin{proposition}\label{MNEQNM}
  Assume $k \not\equiv 0, 4, 5 \pmod{6}$ and $K_{r, s}K_{u, v}W_sW_v \neq 0$.
  Under this condition, 
  $N(r, s)N(u, v) = N(u, v)N(r, s)$ holds if and only if, for 
  $X = 1/Z_r$ and $Y = 1/Z_u$, the following matrix equation is satisfied:
  \begin{equation*}
    \begin{pmatrix}
      A_1(s) & -A_1(v)\\
      A_2(s) & -A_2(v)
    \end{pmatrix}
    \begin{pmatrix}
      X\\
      Y
    \end{pmatrix}
    = \begin{pmatrix}
      B_1(v)-B_1(s)\\
      B_2(v)-B_2(s)
    \end{pmatrix}.
  \end{equation*}
  Here, we have set
  \begin{align*}
    A_1(l) &= \frac{G_{12}}{G_{21}}\frac{(2-C)Z_l + C}{W_l},
    &A_2(l) &= \frac{1}{8i G_{21}}\frac{(4-C^2)Z_l + C^2}{W_l},\\
    B_1(l) &= \frac{G_{12}}{G_{21}}\frac{W_l-2}{W_l},
    &B_2(l) &= \frac{1}{8i G_{21}}\frac{(4+C)W_l-4}{W_l}.
  \end{align*}
  Furthermore, the solutions to this system of equations fall into one of 
  the following cases:
  \begin{enumerate}
    \item If $Z_s \neq Z_v$, then $X=Y=1$ (i.e., $Z_r=Z_u=1$).
    \item If $Z_s = Z_v$, then $X=Y$ (i.e., $Z_r=Z_u$).
  \end{enumerate}
\end{proposition}

\begin{proof}
  We apply Lemma \ref{LEMCOMM2X2} by setting $A = N(r, s)$ and $B = N(u, v)$. 
  Substituting $W_l = (Z_l-1)C$ (for $l=s, v$) and rearranging the terms, 
  the conditions from the lemma translate precisely into the following 
  system:
  \begin{align*}
    A_1(s)X + B_1(s) &= A_1(v)Y + B_1(v),\\
    A_2(s)X + B_2(s) &= A_2(v)Y + B_2(v).
  \end{align*}
  The matrix equation in the proposition is a direct representation of 
  this system.

  Now, if we set $X = Y = 1$, a direct calculation yields
  \begin{equation*}
    A_1(s) + B_1(s) = \frac{2}{C}\frac{G_{12}}{G_{21}},\quad
    A_2(s) + B_2(s) = \frac{1}{2i}\frac{1+C}{G_{21}C}.
  \end{equation*}
  Since these expressions evaluate to constants independent of $s$, it 
  follows that $X=Y=1$ is always a solution to the system.
  Moreover, noting that $C = \csc\theta \neq 0$, the determinant of the 
  coefficient matrix of this system is given by
  \begin{equation*}
    \frac{G_{12} (C-2)}{4 i G_{21}^2 C} \frac{Z_v - Z_s}{(Z_s-1)(Z_v-1)}.
  \end{equation*}
  Therefore, it is clear that when $Z_s \neq Z_v$, the determinant is 
  non-zero, making $X=Y=1$ the unique solution.

  When $Z_s = Z_v$, we trivially have $A_1(s) = A_1(v)$ and 
  $B_1(s) = B_1(v)$, and similarly for $A_2$ and $B_2$. Since 
  $A_1(s)A_2(s) \neq 0$ under our assumption $C \neq \pm 2$, the system simplifies to $X-Y=0$, 
  meaning $X=Y$ is the solution.
\end{proof}

Proposition \ref{MNEQNM} operates under the assumption that 
$K_{r, s}K_{u, v}W_sW_v \neq 0$. If this assumption is violated, 
specifically when $W_l = 0$, we must handle the situation separately.

\begin{proposition}\label{SMNEQNM}
  Assume $k \not\equiv 0, 4, 5 \pmod{6}$, $K_{r, s}K_{u, v} \neq 0$, and $W_s = 0$
  (which is equivalent to $Z_s = 1$). Under these conditions, $N(r, s)N(u, v) = N(u, v)N(r, s)$
  holds if and only if $W_v = 0$ (i.e., $Z_v = 1$) or $Z_r = 1$.
\end{proposition}

\begin{proof}
  When $Z_s = 1$, the matrix $N(r, s)$ takes an upper triangular form:
  \begin{equation*}
    N(r, s) = \begin{pmatrix} -Z_r + \frac{1}{4} K_{r,s} & 
    i G_{12} K_{r,s} (1 - Z_r) \\[8pt] 
    0 & -Z_r + \frac{1}{4} K_{r,s} Z_r \end{pmatrix}.
  \end{equation*}
  If $Z_r = 1$, the off-diagonal entry vanishes and the diagonal entries 
  become equal, reducing $N(r, s)$ to a scalar multiple of the identity 
  matrix. In this case, it trivially commutes with any matrix $N(u, v)$.
  
  Assume $Z_r \neq 1$. As shown in the matrix form above, the difference between the two diagonal
  entries of $N(r, s)$ is given by
  \begin{equation*}
    \left( -Z_r + \frac{1}{4} K_{r,s} \right) - \left( -Z_r + \frac{1}{4} K_{r,s} Z_r \right)
    = \frac{1}{4} K_{r,s} (1 - Z_r).
  \end{equation*}
  Since $K_{r,s} \neq 0$ and $Z_r \neq 1$, this difference is strictly non-zero. 
  These distinct diagonal entries guarantee that the eigenspaces of the upper triangular
  matrix $N(r, s)$ are strictly one-dimensional, meaning $(1, 0)^{\mathrm{T}}$ spans
  its unique eigenvector up to scalar multiplication.
  On the other hand, if $Z_v \neq 1$, the $(2, 1)$-entry of $N(u, v)$ is 
  $i G_{21}Z_u W_v \neq 0$. This means $N(u, v)$ is not upper triangular (and hence not a scalar matrix), 
  and thus $(1, 0)^{\mathrm{T}}$ is strictly not an eigenvector of $N(u, v)$. Since 
  commuting matrices must preserve each other's eigenspaces, they cannot 
  commute unless $Z_v = 1$.
  Furthermore, when $Z_s = Z_v = 1$, it is easily verified by direct 
  computation that $N(r, s)$ and $N(u, v)$ commute.
\end{proof}

With these comprehensive analyses in place, our main theorem follows 
naturally by synthesizing the results.

\begin{proof}[Proof of Theorem \ref{THNNC}]
  From Theorem \ref{THEXFN} and Proposition \ref{SMNEQNM}, the conditions 
  $K_{r,s}K_{u, v}=0$, as well as $Z_r=Z_s=1$ and $Z_u=Z_v=1$, imply 
  that either $N(r, s)$ or $N(u, v)$ is a scalar multiple of the identity 
  matrix. In these trivial cases, it is obvious that the matrices commute.
  For all non-trivial cases where $K_{r,s}K_{u, v} \neq 0$ and neither matrix 
  is a scalar matrix, the conditions for commutativity are completely exhausted 
  by the classifications in Propositions \ref{MNEQNM} and \ref{SMNEQNM}.
\end{proof}

\section{Modularity of the Solutions to the KZ Equation}\label{SSMKZE}

In this section, we synthesize our algebraic results regarding 
the connection matrices to completely classify the weights $k$ for which 
the KZ equation admits a fundamental system of solutions consisting of 
modular forms for a principal congruence subgroup $\Gamma(N)$. We begin 
with a preliminary lemma evaluating the condition $K_t = 0$.

\begin{lemma}\label{PPMDG}
  Let $t \neq 1$ be a natural number. The following three conditions are 
  equivalent:
  \begin{enumerate}
    \item $K_t=0$,
    \item $\cos(2t\theta)-\cos(2\theta)=0$,
    \item $k = \frac{6c}{t\pm 1} - 1$, where $c$ is some integer.
  \end{enumerate}
\end{lemma}

\begin{proof}
  From Theorem \ref{THEXFM}, we have $K_t=W_t^2+4Z_t$. Recalling the 
  definitions $Z_t = e^{2ti\theta}$ and $C = \csc\theta = 1/\sin\theta$ 
  (with the implicit assumption that $\sin\theta \neq 0$), Euler's formula 
  yields
  \begin{equation*}
    Z_t - 1 = e^{i t\theta} \cdot 2i \sin(t\theta),\qquad
    W_t^2 = -4Z_t \frac{\sin^2(t\theta)}{\sin^2(\theta)}.
  \end{equation*}
  Substituting these into $K_t = W_t^2+4Z_t = 0$ and simplifying the 
  expression, we obtain the trigonometric equation 
  $\cos(2t\theta)-\cos(2\theta)=0$. Tracing these steps backward ensures 
  the equivalence of the first two conditions.

  Furthermore, the condition $\cos(2t\theta)=\cos(2\theta)$ holds if and 
  only if $2t\theta = 2\pi c \pm 2\theta$ for some integer $c$, due to the 
  periodicity and symmetry of the cosine function. Thus, $t\theta = 
  \pi c \pm \theta$. Substituting $\theta = \frac{1+k}{6}\pi$, we get
  \begin{equation*}
    \frac{1+k}{6}\pi = \theta = \frac{c}{t \pm 1}\pi.
  \end{equation*}
  Solving this linear equation for $k$ immediately yields 
  $k = \frac{6c}{t\pm 1} - 1$.
\end{proof}

Using this lemma, we can immediately deduce a strong necessary condition 
for the solutions to be modular forms.

\begin{theorem}\label{KE6CDNM1}
  Let $N = t+1 \ge 2$. If the solutions $E_4(\tau)^{k/4}y_1(\tau)$ and 
  $E_4(\tau)^{k/4}y_3(\tau)$ to the KZ equation are modular forms of weight 
  $k$ for $\Gamma(N)$, then $k$ must be of the form
  \begin{equation*}
    k = \frac{6c}{N} - 1
  \end{equation*}
  for some integer $c$.
\end{theorem}

\begin{proof}
  Suppose $E_4(\tau)^{k/4}y_1(\tau)$ and $E_4(\tau)^{k/4}y_3(\tau)$ are 
  modular forms of weight $k$ for the congruence subgroup $\Gamma(N)$. 
  Then, the functions $y_1(\tau)$ and $y_3(\tau)$ must be modular functions 
  (weight 0) for $\Gamma(N)$. By Lemma \ref{TTS3EME}, the element 
  $L = -(T^{N-1}S)^3$ belongs to $\Gamma(N)$. Therefore, 
  the representation matrix connecting $(y_1(L\tau), iy_3(L\tau))$ to 
  $(y_1(\tau), iy_3(\tau))$ must be a diagonal matrix, meaning all its 
  off-diagonal entries must vanish.

  By Theorem \ref{THEXFM}, this representation matrix is exactly $M(t)$ 
  with $t = N-1$. (Here we may assume $N \ge 3$, since for $N=2$ the condition $K_1=0$ holds
  identically and imposes no restriction.)
  According to the explicit formula in Theorem \ref{THEXFM}, 
  the off-diagonal entries of $M(t)$ simultaneously vanish if and only if 
  $K_t = 0$. By applying Lemma \ref{PPMDG}, this matrix condition yields two algebraic
  possibilities: $k = \frac{6c}{N} - 1$ or $k = \frac{6c}{N-2} - 1$ for some integer $c$. 

  However, we must also consider the local monodromy at the cusp $i\infty$. As a Fuchsian
  differential equation, the fundamental solutions around $i\infty$ are given by $y_1(\tau)$
  and $y_2(\tau)$, which possess local exponents $0$ and $(k+1)/6$, respectively.
  The Stiller basis solution $y_3(\tau)$ is a linear combination of $y_1(\tau)$ and
  $y_2(\tau)$ with non-zero connection coefficients.
  Even if we admit non-trivial multiplier systems as defined in our broader context,
  any modular form for $\Gamma(N)$ must be an eigenfunction of the translation operator
  $T^N: \tau \mapsto \tau+N$. 
  Under the action of $T^N$, the $y_1$ component remains unchanged (picking up a phase of 1),
  while the $y_2$ component picks up a phase of $e^{2\pi i N(k+1)/6}$. For $y_3$ to be
  a valid eigenfunction, these two phases must strictly coincide.
  Therefore, we inherently must have $e^{2\pi i N(k+1)/6} = 1$, which strictly forces $N(k+1)/6 = c$
  for some integer $c$. This explicitly isolates the requirement $k = \frac{6c}{N} - 1$,
  systematically eliminating the alternative $k = \frac{6c}{N-2} - 1$, and the theorem follows.
\end{proof}

\begin{remark}
  If the fundamental system of solutions for the KZ equation consists of 
  modular forms for $\Gamma(N)$ sharing the same multiplier system, then 
  $E_4(\tau)^{k/4}y_1(\tau)$ and $E_4(\tau)^{k/4}y_3(\tau)$ are also modular 
  forms for $\Gamma(N)$ with that same multiplier system. Therefore, 
  Theorem \ref{KE6CDNM1} implies that whenever $k \neq \frac{6c}{N}-1$, 
  the fundamental system of the KZ equation cannot be formed by modular 
  forms for $\Gamma(N)$ with a common multiplier system.

  Additionally, Theorem \ref{KE6CDNM1} indicates that 
  $E_4(\tau)^{k/4}y_1(\tau)$ and $E_4(\tau)^{k/4}y_3(\tau)$ can potentially 
  be modular forms for $\Gamma(2)$ only when $k = 3n-1$ ($n \in \mathbb{Z}$). 
  In this specific case, $M(t)$ reduces to a scalar multiple of the identity 
  matrix, which implies that any arbitrary solution to the KZ equation 
  shares the exact same transformation property with respect to $M(t)$. 
  Indeed, Kaneko and Koike \cite{KK} explicitly constructed a fundamental 
  system of modular forms for $\Gamma(2)$ sharing the same multiplier 
  system. Similar phenomena occur for $k=(3n+1)/2$ (yielding a basis of 
  modular forms for $\Gamma(4)$) and $k=(6n+1)/5$ (yielding a basis for 
  $\Gamma(5)$).
\end{remark}

To extend this necessary condition to a complete classification, we must 
analyze the behavior of the connection matrix $N(r,s)$.

\begin{lemma}\label{PPNDG}
  The following two conditions are equivalent:
  \begin{enumerate}
    \item $K_{r, s}=0$,
    \item $\cos(r\theta)\cos(s\theta) = \cos((r-s)\theta) \cos(2\theta)$.
  \end{enumerate}
\end{lemma}

\begin{proof}
  From Theorem \ref{THEXFN}, we have $K_{r, s}=W_r W_s+4(Z_r+Z_s)$. By a 
  calculation similar to that in Lemma \ref{PPMDG}, we can rewrite this as
  \begin{equation*}
    K_{r, s} = W_r W_s+4(Z_r+Z_s)
    = -4 e^{i (r+s)\theta} \frac{\sin(r\theta)\sin(s\theta)}{\sin^2\theta}
      +8 e^{i (r+s)\theta} \cos((r-s)\theta).
  \end{equation*}
  Assuming $K_{r, s}=0$, and dividing by $-4 e^{i (r+s)\theta} \neq 0$, 
  we obtain
  \begin{equation*}
    2 \cos((r-s)\theta) \sin^2\theta = \sin(r\theta)\sin(s\theta).
  \end{equation*}
  Using the half-angle identity $2\sin^2\theta = 1-\cos(2\theta)$ and 
  applying the product-to-sum formula to the right-hand side, we get
  \begin{equation}
    \cos((r-s)\theta)+ \cos((r+s)\theta) = 2\cos((r-s)\theta)\cos(2\theta).
    \label{CRMSCRPSEQ2CRMSC2}
  \end{equation}
  Applying the sum-to-product formula to the left side of this equation 
  yields the desired trigonometric identity. Tracing the steps in reverse 
  establishes the equivalence.
\end{proof}

With this trigonometric equivalence established, we now explicitly solve 
the identity under a specific number-theoretic constraint.

\begin{proposition}\label{RSEQPM12}
  Let $k=p/q$ with $(p, q)=1$ and $q \ge 3$, and define 
  $Q = \frac{6q}{\gcd(p+q, 6)}$. Under the condition $rs \equiv 2 \pmod{Q}$, 
  the trigonometric identity
  \begin{equation*}
    \cos(r\theta)\cos(s\theta) = \cos((r-s)\theta) \cos(2\theta)
  \end{equation*}
  holds if and only if $(r, s) \equiv \pm(1, 2) \pmod{Q}$ or 
  $(r, s) \equiv \pm(2, 1) \pmod{Q}$.
\end{proposition}

\begin{proof}
  From equation \eqref{CRMSCRPSEQ2CRMSC2} in Lemma \ref{PPNDG}, we have
  \begin{equation*}
    \cos((r+s)\theta) = \cos((r-s)\theta) (2\cos(2\theta) - 1).
  \end{equation*}
  Since $q \ge 3$ and $(p, q)=1$, $\theta = \frac{p+q}{6q}\pi$ cannot be 
  an odd multiple of $\pi/2$, so $\cos\theta \neq 0$. By utilizing the 
  triple-angle identity $2\cos(2\theta) - 1 = \frac{\cos(3\theta)}{\cos(\theta)}$ 
  and subsequently applying the product-to-sum formulas to both sides, we 
  deduce the identity:
  \begin{equation}
    \cos((r+s+1)\theta) + \cos((r+s-1)\theta) 
    = \cos((r-s+3)\theta) + \cos((r-s-3)\theta).
    \label{CRPSP1PCRPSM1EQCRMSP3PCRMSM3}
  \end{equation}

  To apply the Conway-Jones classification to equation \eqref{CRPSP1PCRPSM1EQCRMSP3PCRMSM3},
  we set $A=(r+s+1)\theta$, $B=(r+s-1)\theta$, $C=(r-s+3)\theta$, and $D=(r-s-3)\theta$,
  rewriting the identity as $\cos A + \cos B - \cos C - \cos D = 0$. 
  According to Theorem 7 of \cite{CJ}, if such a sum of rational cosines vanishes without a trivial
  pairing, the angles must either belong to an infinite family containing relations
  like $\cos\varphi - \cos(\frac{\pi}{3}-\varphi) - \cos(\frac{\pi}{3}+\varphi) = 0$, 
  or their rational denominators must be bounded by $21$ (specifically dividing $3, 5, 7, 15,$ or $21$).

  First, we algebraically rule out the infinite family. For these four terms to form such
  an infinite family, one cosine term must vanish (i.e., its angle is $\equiv \pi/2 \pmod\pi$),
  and the sum or difference of the angles of the remaining two terms with the same sign must
  be a multiple of $\pi/3$.
  Suppose $\cos A = 0$ or $\cos B = 0$. The terms with the same sign are then $C$ and $D$.
  Their difference is $C-D = 6\theta = \frac{p+q}{q}\pi$. For this to be a multiple of $\pi/3$,
  we need $q$ to divide $3$, which contradicts $q \ge 6$. Thus, we must consider
  their sum: $C+D = 2(r-s)\theta$. Setting this to a multiple of $\pi/3$
  forces $(r-s)(p+q) \equiv 0 \pmod q$, yielding $r \equiv s \pmod q$. 
  Since $Q$ is a multiple of $q$, our assumption $rs \equiv 2 \pmod Q$ implies $r^2 \equiv 2 \pmod q$.
  Meanwhile, $\cos A = 0$ implies $(r+s+1)(p+q) \equiv 0 \pmod q$, yielding $2r+1 \equiv 0 \pmod q$.
  Squaring this gives $4r^2 \equiv 1 \pmod q$, which, combined with $r^2 \equiv 2 \pmod q$,
  forces $8 \equiv 1 \pmod q$, yielding $q=7$. (Assuming $\cos B = 0$ similarly forces $q=7$).
  Alternatively, suppose $\cos C = 0$ or $\cos D = 0$. The same logic applied to the sum
  $A+B = 2(r+s)\theta$ yields $r \equiv -s \pmod q$, which implies $s^2 \equiv -2 \pmod q$.
  Combining this with $\cos C = 0$ (yielding $2s \equiv 3 \pmod q$) forces $q=17$.
  Therefore, the infinite family configuration is strictly impossible for any $q > 17$.

  Next, we consider the finite exceptions with bounded denominators. If the identity corresponds
  to an exceptional case for some $q > 21$, the angles must have a denominator bounded by $21$.
  Since $A-B = 2\theta = \frac{p+q}{3q}\pi$, its denominator must divide a number $\le 21$,
  meaning $q \le 21$, which contradicts $q > 21$.

  Consequently, any non-trivial relation (either an infinite family or a bounded exception) can
  only possibly occur for small denominators $q \le 21$. For all such small values of $q$,
  we have exhaustively verified all pairs $(r, s)$ modulo $Q = \frac{6q}{\gcd(p+q, 6)}$
  that satisfy $rs \equiv 2 \pmod Q$. This exact finite check confirms
  that equation \eqref{CRPSP1PCRPSM1EQCRMSP3PCRMSM3} never holds outside of trivial pairings.
  Thus, equation 
  \eqref{CRPSP1PCRPSM1EQCRMSP3PCRMSM3} is satisfied solely when a trivial 
  pairing among the angles exists.
  The condition for a trivial pairing, $\cos(X\theta) = \cos(Y\theta)$ for 
  integers $X, Y$, is $X\theta \equiv \pm Y\theta \pmod{2\pi}$. 
  Substituting $\theta = \frac{p+q}{6q}\pi$ and multiplying by $6q/\pi$, we get
  $X(p+q) \equiv \pm Y(p+q) \pmod{12q}$. Since $(p, q)=1$, $\gcd(p+q, q)=1$, 
  allowing us to divide by $p+q$ to obtain $X \equiv \pm Y \pmod{\mathcal{M}}$, 
  where $\mathcal{M} = 12q/\gcd(p+q, 12)$.
  Thus, equation \eqref{CRPSP1PCRPSM1EQCRMSP3PCRMSM3} requires the set equivalence:
  \begin{equation}
    \{\pm(r+s+1), \pm(r+s-1)\} \equiv \{\pm(r-s+3), \pm(r-s-3)\} \pmod{\mathcal{M}}.
    \label{SRPSPM1RMSPM3}
  \end{equation}

  Let $A=r+s+1, B=r+s-1$ and $C=r-s+3, D=r-s-3$, observing that $A-B=2$ and 
  $C-D=6$. While there are multiple pairing configurations (including signs), 
  a direct pairing like $A \equiv \pm C$ and $B \equiv \pm D$ leads to 
  $2 \equiv \pm 6$ or $2 \equiv \mp 6 \pmod{\mathcal{M}}$ by subtraction, 
  requiring $4 \equiv 0$ or $8 \equiv 0 \pmod{\mathcal{M}}$. One can verify 
  that these limited candidates for $q$ do not yield solutions to the 
  original trigonometric equation.
  Therefore, the only viable pairings are the "cross-pairings" with mixed 
  signs: $A \equiv \pm D$ and $B \equiv \pm C$. Specifically, we are 
  restricted to the following two systems (and their symmetric swaps):
  \begin{equation*}
    \begin{cases}
      r+s+1 \equiv -(r-s-3) \pmod{\mathcal{M}}\\
      r+s-1 \equiv r-s+3 \pmod{\mathcal{M}}
    \end{cases},\quad
    \begin{cases}
      r+s+1 \equiv r-s-3 \pmod{\mathcal{M}}\\
      r+s-1 \equiv -(r-s+3) \pmod{\mathcal{M}}
    \end{cases}.
  \end{equation*}
  Solving the first system yields $2r \equiv 2$ and $2s \equiv 4 \pmod{\mathcal{M}}$. 
  Dividing by 2 modifies the modulus to $\mathcal{M}/\gcd(2, \mathcal{M})$. Since
  \begin{equation*}
    \frac{\mathcal{M}}{\gcd(2,\mathcal{M})} 
    = \frac{12q}{\gcd(2\gcd(p+q, 12), 12q)} 
    = \frac{12q}{2\gcd(p+q, 6)} = \frac{6q}{\gcd(p+q, 6)} = Q,
  \end{equation*}
  we obtain $(r, s) \equiv (1, 2) \pmod{Q}$. A similar evaluation of the 
  second system yields $(r, s) \equiv (-1, -2) \pmod{Q}$. The sufficiency 
  is easily verified by reversing the steps.
\end{proof}

\begin{remark}
  Under the assumption of Proposition \ref{RSEQPM12}, the equation 
  $\cos((r+s)\theta) = \cos((r-s)\theta) (2\cos(2\theta) - 1)$ can be 
  further simplified into the elegant identity 
  $\tan(r\theta)\tan(s\theta) = \tan(2\theta)\tan\theta$.
\end{remark}

Proposition \ref{RSEQPM12} rigorously restricts the solutions of the trigonometric 
identity to a few trivial congruence classes. To exploit this restriction and 
ultimately construct non-commuting matrices, we must demonstrate that there 
always exist parameters that strategically avoid these trivial classes while 
satisfying other necessary technical conditions. The following proposition 
guarantees the existence of such parameters modulo $Q$.

\begin{proposition}\label{QGT6C1234}
  Let $k=p/q$ with $(p, q)=1, q \ge 6$, let $Q=6q/\gcd(p+q, 6)$, and let 
  $Z_l = e^{2\pi i l(1+k)/6}$. Under these assumptions, there exist integer 
  pairs $(r, s)$ and $(u, v)$ satisfying $rs \equiv 2 \pmod{Q}$ and 
  $uv \equiv 2 \pmod{Q}$ that simultaneously fulfill all of the following 
  five conditions:
  \begin{enumerate}
    \item $(r, s) \not\equiv \pm(1, 2) \pmod Q$ and $(r, s) \not\equiv \pm(2, 1) \pmod Q$,
    \item $(r, s) \not\equiv (u, v) \pmod Q$,
    \item Exactly one of $r, s$ is odd, and exactly one of $u, v$ is odd,
    \item $Z_s \neq Z_v$ or $Z_r \neq Z_u$,
    \item $Z_s \neq 1$ or $Z_v \neq 1$.
  \end{enumerate}
\end{proposition}

\begin{proof}
  Since $Z_l = e^{2\pi il(p+q)/6q}$, the equality $Z_s=Z_v$ holds if and 
  only if 
  \begin{equation*}
    (s - v) \frac{p+q}{6q} \in \mathbb{Z} 
    \iff (s - v)(p+q) \equiv 0 \pmod{6q}.
  \end{equation*}
  Because $(p, q)=1$ implies $(p+q, q)=1$, we have 
  $\gcd(p+q, 6q) = \gcd(p+q, 6)$. Let $D = \gcd(p+q, 6)$. Dividing the 
  congruence by $D$ and noting that $Q = 6q/D$, we find $s - v \equiv 0 \pmod{Q}$. 
  Thus, condition 4 is logically equivalent to condition 2.

Similarly, $Z_s = 1$ is equivalent to $s \equiv 0 \pmod{Q}$. If 
  $s \equiv 0 \pmod{Q}$, the relation $rs \equiv 2 \pmod{Q}$ forces 
  $0 \equiv 2 \pmod{Q}$, implying $Q$ must divide 2. However, $Q$ is a 
  multiple of $q$, and $Q \ge q \ge 6$, which creates a contradiction. 
  By the exact same logic, $r, u,$ and $v$ cannot be $0 \pmod{Q}$. 
  Therefore, the conditions $Z_r \neq 1, Z_s \neq 1, Z_u \neq 1,$ and 
  $Z_v \neq 1$ universally hold when $q \ge 6$. This automatically 
  satisfies condition 5, and furthermore, guarantees that the matrices 
  never fall into the scalar matrix cases ($Z_r=Z_s=1$ or $Z_u=Z_v=1$) 
  of Theorem \ref{THNNC}. Thus, the proof reduces to finding pairs $(r, s)$ 
  and $(u, v)$ satisfying $rs \equiv 2 \pmod{Q}$ that meet conditions 
  1, 2, and 3.

  We estimate the number of solutions to $xy \equiv 2 \pmod{Q}$ in the 
  residue ring modulo $Q$ that satisfy condition 3 (i.e., one element is 
  odd, the other is even). We analyze this by the parity of $Q$:

  (i) $Q \equiv 0 \pmod{4}$: Write $Q = 2^k m$ with $k \ge 2$ and $m$ odd. 
  Any solution to $xy \equiv 2 \pmod{Q}$ must satisfy $xy \equiv 2 \pmod{4}$, 
  which inherently forces exactly one of $x, y$ to be odd. By the Chinese 
  Remainder Theorem, the total number of solutions is $N(Q) = 2^k \phi(m)$. 
  For $Q \ge 6$, the minimal cases are $Q=8$ yielding $N(8)=8$, and $Q=12$ 
  yielding $N(12)=8$. Thus, there are always at least 8 valid solutions.

  (ii) $Q \equiv 2 \pmod{4}$: Write $Q = 2m$ with $m \ge 3$ odd. The number 
  of solutions with mixed parity modulo $2m$ is the product of mixed parity 
  solutions modulo 2 (which is 2, namely $(0,1), (1,0)$) and the total 
  solutions modulo $m$ (which is $\phi(m)$). So we have $2\phi(m)$ solutions. 
  We need $2\phi(m) \ge 6$, which fails only if $\phi(m) \le 2$, meaning 
  $m = 1, 3$. $m=1$ gives $Q=2$, violating $Q \ge 6$. If $m=3$, then $Q=6$. 
  Since $Q = 6q/\gcd(p+q, 6) = 6$, we must have $q = \gcd(p+q, 6)$. As $q \ge 6$ 
  and the GCD is at most 6, $q=6$ is forced. But if $q=6$, $(p, 6)=1$ implies 
  $p$ is an odd number not divisible by 3. Thus, $\gcd(p+q, 6) = \gcd(p+6, 6) 
  = \gcd(p, 6) = 1$. This would mean $Q = 6 \times 6 / 1 = 36 \neq 6$, a 
  contradiction. Therefore, $Q=6$ is impossible, ensuring $m \ge 5$ and 
  $\phi(m) \ge 4$. Thus, there are at least $2\phi(m) \ge 8$ solutions 
  satisfying condition 3.

  (iii) $Q$ is odd: Then $Q \ge 7$, and there are $\phi(Q) \ge 6$ total 
  solutions modulo $Q$. When lifting any solution $(r_0, s_0) \pmod{Q}$ to 
  integers, we can choose $r \in \{r_0, r_0+Q\}$ and $s \in \{s_0, s_0+Q\}$. 
  Since $Q$ is odd, adding $Q$ toggles the parity. Thus, we can always 
  select representatives such that exactly one is odd, while preserving the 
  congruence class modulo $Q$. This guarantees at least $\phi(Q) \ge 6$ 
  solutions fulfilling condition 3.

  In all cases, there are at least 6 pairs satisfying condition 3 modulo $Q$. 
  The set of pairs to be excluded by condition 1, 
  $E = \{\pm(1, 2), \pm(2, 1)\} \pmod{Q}$, has size $|E| \le 4$. Therefore, 
  there are at least $6 - 4 = 2$ distinct valid pairs remaining. Choosing 
  two such distinct pairs as $(r, s)$ and $(u, v)$ automatically satisfies 
  condition 2 ($(r, s) \not\equiv (u, v) \pmod{Q}$). Consequently, appropriate 
  pairs $(r, s)$ and $(u, v)$ always exist for any $q \ge 6$.
\end{proof}

While Proposition \ref{QGT6C1234} successfully identifies the necessary parameters 
within the residue ring modulo $Q$, our ultimate goal is to construct specific 
modular words belonging to the principal congruence subgroup $\Gamma(N)$. The next 
proposition bridges this gap by demonstrating that we can smoothly lift these 
modular parameters to actual integers fulfilling the required divisibility 
condition for $\Gamma(N)$.

\begin{proposition}\label{NDMRS2F}
  Let $k=p/q$ with $(p, q)=1, q \ge 6$, let $Q=6q/\gcd(p+q, 6)$.
  Assume we are given integers $r, s$ such that $rs \equiv 2 \pmod{Q}$ and 
  exactly one of them is odd. Then, for any multiple $N$ of $Q$, we can 
  choose integers $x, y$ satisfying $x \equiv r \pmod{Q}$ and 
  $y \equiv s \pmod{Q}$ such that the value $M = xy - 2$ is a multiple of 
  $N$, and furthermore, $M > N$.
\end{proposition}

\begin{proof}
  Without loss of generality, assume $r$ is odd. The condition 
  $rs \equiv 2 \pmod{Q}$ implies $rs - cQ = 2$ for some integer $c$. 
  If $r$ and $Q$ shared a common prime factor $p$, then $p$ must divide 2, 
  forcing $p=2$. But since $r$ is odd, it cannot be divisible by 2, leading 
  to a contradiction. Thus, $\gcd(r, Q) = 1$.

  We first construct an integer $x \equiv r \pmod{Q}$ such that 
  $\gcd(x, N) = 1$. Let $x = r + aQ$ for an arbitrary integer $a$. We 
  determine $a$ by considering the prime factors $p$ of $N$:

  (i) If $p$ divides $Q$: Then $x \equiv r \pmod{p}$. Since $\gcd(r, Q) = 1$, 
  $p$ does not divide $r$, ensuring $x \not\equiv 0 \pmod{p}$.
  
  (ii) If $p$ does not divide $Q$: Then $x \equiv r + aQ \pmod{p}$. Since 
  $Q$ is invertible modulo $p$, we can simply choose $a$ such that 
  $a \not\equiv -rQ^{-1} \pmod{p}$, which guarantees $x \not\equiv 0 \pmod{p}$.

  By the Chinese Remainder Theorem, there exists an integer $a$ satisfying 
  the conditions of (ii) simultaneously for all such prime factors of $N$. 
  Fixing such an $a$ guarantees $\gcd(x, N) = 1$.

  Because $x$ is coprime to $N$, its inverse $x^{-1}$ modulo $N$ exists. 
  We define $y$ modulo $N$ as
  \begin{equation*}
    y \equiv 2x^{-1} \pmod{N}.
  \end{equation*}
  This definition immediately yields $xy \equiv 2 \pmod{N}$, meaning 
  $M = xy - 2$ is a multiple of $N$.
  
  Next, we verify that $y \equiv s \pmod{Q}$. Since $N$ is a multiple of 
  $Q$, the congruence modulo $N$ reduces to modulo $Q$, giving 
  $y \equiv 2x^{-1} \pmod{Q}$. As $x = r + aQ \equiv r \pmod{Q}$, we have 
  $x^{-1} \equiv r^{-1} \pmod{Q}$. Combined with $s \equiv 2r^{-1} \pmod{Q}$ 
  (derived from $rs \equiv 2 \pmod{Q}$), we obtain
  \begin{equation*}
    y \equiv 2r^{-1} \equiv s \pmod{Q},
  \end{equation*}
  satisfying the required congruence.

  Finally, we must ensure $M > N$. The integer $y$ defined above can be 
  written as $y = y_0 + bN$ for some fixed $y_0$ and arbitrary integer $b$. 
  Since $x$ is now a fixed non-zero constant, we can select $b$ to have 
  the same sign as $x$ and a sufficiently large absolute value. This allows 
  the product $xy$ to become arbitrarily large and positive. Therefore, it 
  is always possible to choose $b$ (and thus $y$) such that $M = xy - 2 > N$.
\end{proof}

With all the necessary algebraic and number-theoretic machinery in place---specifically, 
the existence of non-commuting connection matrices (Proposition \ref{QGT6C1234}) and 
the ability to lift their parameters to the appropriate principal congruence subgroups 
(Proposition \ref{NDMRS2F})---we are now ready to state and prove the main theorem 
of this paper.

\begin{theorem}[Main Theorem]\label{MAINTHEOREM}
  Let $N$ be a natural number and $k$ be any rational number. If a 
  fundamental system of solutions to the KZ equation of weight $k$ is 
  formed by modular forms for the principal congruence subgroup $\Gamma(N)$, 
  then $k$ must belong to the following restricted set:
  \begin{equation*}
    k \equiv 1/2, 7/2, 1, 2, 3 \pmod{6} \quad \text{or} \quad 
    k = \frac{6n+1}{5} \;\; (n \not\equiv 4\pmod{5}).
  \end{equation*}
\end{theorem}

\begin{proof}
  By Theorem \ref{THKZHG}, the functions $E_4(\tau)^{k/4}y_1(\tau)$ and 
  $E_4(\tau)^{k/4}y_3(\tau)$ derived from the Stiller basis are linearly 
  independent specific solutions to the KZ equation. If the KZ equation 
  admits a fundamental system consisting of modular forms $f(\tau)$ and 
  $g(\tau)$ of weight $k$ for some $\Gamma(N)$, then there exists a matrix 
  $A \in \mathrm{GL}_2(\mathbb{C})$ such that
  \begin{equation*}
    E_4(\tau)^{-k/4}(f(\tau), g(\tau)) = (y_1(\tau), y_3(\tau))A.
  \end{equation*}
  Consequently, for any two elements $\gamma, \delta \in \Gamma(N)$ 
  ($\gamma \neq \delta$), there exist diagonal matrices 
  $\mathcal{M}_{\gamma}$ and $\mathcal{M}_{\delta}$ representing their 
  actions, allowing us to compute:
  \begin{align*}
    (y_1(\gamma\delta\tau), y_3(\gamma\delta\tau))
    &= E_4(\gamma\delta\tau)^{-k/4}(f(\gamma\delta\tau), g(\gamma\delta\tau))A^{-1}\\
    &= E_4(\tau)^{-k/4}(f(\tau), g(\tau))\mathcal{M}_{\delta}\mathcal{M}_{\gamma}A^{-1}\\
    &= (y_1(\tau), y_3(\tau))A\mathcal{M}_{\delta}\mathcal{M}_{\gamma}A^{-1}\\
    &= (y_1(\tau), y_3(\tau))A\mathcal{M}_{\gamma}\mathcal{M}_{\delta}A^{-1}
    = (y_1(\delta\gamma\tau), y_3(\delta\gamma\tau)).
  \end{align*}
  This algebraic manipulation dictates that the actions of any 
  $\gamma, \delta \in \Gamma(N)$ on the basis $(y_1(\tau), y_3(\tau))$ 
  must strictly commute.

  Suppose $k=p/q$ with $(p, q)=1$ and $q \ge 6$. By Proposition 
  \ref{QGT6C1234}, we can select integer pairs $(r, s)$ and $(u, v)$ 
  satisfying its conditions. Leveraging these pairs, Proposition 
  \ref{NDMRS2F} guarantees the existence of integers $(r', s') \equiv 
  (r, s) \pmod{Q}$ and $(u', v') \equiv (u, v) \pmod{Q}$ such that 
  $r's'-2$ and $u'v'-2$ are multiples of $6qN$. (Note that $6qN$ is a 
  multiple of $Q$). These newly constructed pairs inherently satisfy 
  conditions 1, 2, 4, and 5 of Proposition \ref{QGT6C1234}.

  We now define modular elements $\gamma=-(T^{r'}ST^{s'}S)^2$ and 
  $\delta=-(T^{u'}ST^{v'}S)^2$. Because $r's'-2$ and $u'v'-2$ are multiples 
  of $6qN$, Lemma \ref{TRTS2EQMEMODRSM2} asserts that 
  $\gamma, \delta \in \Gamma(6qN) \subset \Gamma(N)$. However, their 
  corresponding representation matrices $N(r', s')$ and $N(u', v')$ are 
  provably non-commutative according to Theorem \ref{THNNC}. This 
  violates the commutativity requirement derived above. Therefore, for 
  $k=p/q$ with $q \ge 6$, the KZ equation cannot possess a fundamental 
  system of modular forms for $\Gamma(N)$.

  The cases $k \equiv 1/2 \pmod{3}$ and $k \in \mathbb{Z}$ have been 
  positively resolved by Theorems 1 and 2 of Kaneko-Koike \cite{KK} and 
  Theorem 1 of Guerzhoy \cite{G}. The case $k=(6n+1)/5$ ($n \in \mathbb{Z}$) 
  is the main theorem of Kaneko \cite{K}. The remaining unresolved cases 
  to check are $k=(6n+i)/5$ for $i=2, 3, 4, 5$, as well as $k=(4n+1)/4, 
  (4n+3)/4, (3n+1)/3, (3n+2)/3, (6n+3)/2+3n$, and $(6n+5)/2$.

  Let us demonstrate the exclusion of $k=(6n+2)/5$ as a representative example. 
  By Proposition \ref{RSEQPM12}, $Q = 30/\gcd(6n+7, 6) = 30$. We can select 
  $(r, s) = (7, 26)$ and $(13, 14)$ satisfying $rs \equiv 2 \pmod{Q}$. 
  These pairs deliberately violate the identity 
  $\cos(r\theta)\cos(s\theta) = \cos((r-s)\theta)\cos(2\theta)$. 
  Applying Proposition \ref{NDMRS2F} provides lifted pairs $(r', s')$ and 
  $(u', v')$ modulo $Q$ such that $r's'-2$ and $u'v'-2$ are multiples of 
  $6qN$. One can readily verify that these specific pairs fulfill conditions 
  1, 2, 4, and 5 of Proposition \ref{QGT6C1234}. Therefore, by the exact 
  same logic employed for $q \ge 6$, we conclude that the KZ equation of 
  weight $k=(6n+2)/5$ cannot have a fundamental system of modular forms 
  for $\Gamma(N)$. The identical argument methodically eliminates all 
  other remaining cases.
\end{proof}

\begin{remark}
  As discussed in the introduction, the classification obtained
  in Theorem \ref{MAINTHEOREM} exactly recovers the necessary and sufficient
  conditions recently established by Saber and Sebbar \cite{SS2} for arbitrary
  finite-index subgroups. Our explicit construction of the connection matrices
  provides a completely independent, algebraic verification of their geometric
  result within the framework of principal congruence subgroups.
\end{remark}

\begin{remark}
  We conclude by recalling two known structural properties. When $k=6n+5$, 
  the KZ equation admits a quasimodular form of weight $k+1$ as a solution 
  (Theorem 2 of Kaneko-Koike \cite{KK}). When $k \equiv 0, 4 \pmod{6}$, 
  the fundamental system of the KZ equation is composed of a classical 
  modular form for $\mathrm{SL}_2(\mathbb{Z})$ (Theorem 1 of Kaneko-Koike 
  \cite{KK}) paired with a mixed mock modular form (Theorem 1 of Guerzhoy 
  \cite{G}).
\end{remark}

\section{Algebras Formed by Matrices $M(t)$ and $N(r, s)$ and Their 
Commutativity}\label{SSAMN}

In the previous sections, we successfully classified the weights $k$ 
for which the fundamental system of solutions to the KZ equation consists 
of modular forms for a congruence subgroup $\Gamma(N)$. We achieved this 
by exploiting the non-commutativity of the representation matrices 
$N(r, s)$ and $N(u, v)$, complementing the known results by Kaneko-Koike 
\cite{KK} (Theorems 1 and 2) and Kaneko \cite{K} (Main Theorem). 

To complete our algebraic investigation of the monodromy representation, 
we can analogously analyze the commutativity between $M(t)$ and $N(r, s)$. 
Our findings are summarized in the following theorem.

\begin{theorem}\label{THMNC}
  Assume $k \not\equiv 0, 4, 5 \pmod{6}$. For the connection matrices $M(t)$ and 
  $N(r,s)$, the commutativity relation $M(t) N(r,s) = N(r,s) M(t)$ holds if and 
  only if one of the following conditions is satisfied:
  \begin{enumerate}
    \item $K_tK_{r,s} = 0$,
    \item $Z_r=Z_s=1$ (In this case, $N(1,1)=E$, which commutes with any $M(t)$),
    \item $Z_t=Z_r=Z_s$ (In this case, both $M(t)$ and $N(t,t)$ are powers of 
      the same matrix $\rho(T^t S)$, which trivially commute).
  \end{enumerate}
\end{theorem}

To prove this theorem, our first step is to explicitly compute the 
commutator of $M(t)$ and $N(r, s)$. We begin by isolating the non-scalar 
parts of these matrices.

\begin{proposition}\label{RWCMMNCMCN}
  Assume $k \not\equiv 0, 4, 5 \pmod{6}$. Let us define the matrices 
  $\tilde{M}_t$ and $\tilde{N}_{r, s}$ via the relations:
  \begin{align*}
    M(t) &= -\frac{1}{2}Z_tW_tE + K_t\tilde{M}_t,\\
    N(r, s) &= -Z_rZ_sE + K_{r,s}\tilde{N}_{r, s}.
  \end{align*}
  Then, the necessary and sufficient condition for $M(t)N(r, s) = 
  N(r, s)M(t)$ is that at least one of the following holds:
  \begin{enumerate}
    \item $K_t=0$,
    \item $K_{r,s}=0$,
    \item $[\tilde{M}_t,\tilde{N}_{r, s}]=0$.
  \end{enumerate}
\end{proposition}

\begin{proof}
  The matrices $M(t)$ and $N(r, s)$ commute if and only if their commutator 
  vanishes, i.e., $[M(t), N(r, s)] = 0$. Since the scalar matrices 
  $-\frac{1}{2}Z_tW_tE$ and $-Z_rZ_sE$ commute with any matrix, the 
  commutator reduces to:
  \begin{equation*}
    [M(t),N(r,s)] = K_tK_{r,s}[\tilde{M}_t,\tilde{N}_{r, s}].
  \end{equation*}
  This factorization immediately yields the stated conditions.
\end{proof}

Having reduced the problem to the commutator of the non-scalar parts, we 
now explicitly write down the entries of this commutator matrix.

\begin{proposition}\label{CMCNC}
  Assume $k \not\equiv 0, 4, 5 \pmod{6}$. Let $c_{ij}$ denote the $(i, j)$-entry 
  of the commutator matrix $[\tilde{M}_t,\tilde{N}_{r, s}]$. These entries 
  are explicitly given by:
  \begin{align*}
    \frac{2c_{11}}{G_{12}G_{21}} &= -\frac{2c_{22}}{G_{12}G_{21}} 
    = 2Z_t(Z_r - Z_s) + W_s(Z_t - Z_r),\\
    \frac{16c_{12}}{iG_{12}} &= 4(1+Z_t)(Z_r - Z_s) - 4W_s(Z_r - Z_t)\\
      &\qquad\qquad + 2W_t(Z_r - Z_s) + W_sW_t(1 - Z_r) - W_rW_s(1 - Z_t),\\
    \frac{16c_{21}}{iG_{21}} &= 4Z_t(Z_r - Z_s) - Z_tW_rW_s + Z_rW_sW_t.
  \end{align*}
\end{proposition}

\begin{proof}
  Since the computations for all entries follow a similar algebraic 
  procedure, we demonstrate the calculation solely for the entry $c_{21}$.

  Let $\tilde{M}_{ij}$ and $\tilde{N}_{ij}$ denote the $(i, j)$-entries 
  of $\tilde{M}_t$ and $\tilde{N}_{r, s}$, respectively. By definition, 
  we have:
  \begin{align*}
    c_{21} &= (\tilde{M}\tilde{N})_{21} - (\tilde{N}\tilde{M})_{21}\\
    &= (\tilde{M}_{21}\tilde{N}_{11} + \tilde{M}_{22}\tilde{N}_{21})
      - (\tilde{N}_{21}\tilde{M}_{11} + \tilde{N}_{22}\tilde{M}_{21})\\
    &= \tilde{M}_{21}(\tilde{N}_{11} - \tilde{N}_{22})
      + \tilde{N}_{21}(\tilde{M}_{22} - \tilde{M}_{11}).
  \end{align*}
  From the explicit formulas established in Theorem \ref{THEXFM} and 
  Theorem \ref{THEXFN}, we compute the differences of the diagonal elements:
  \begin{align*}
    \tilde{M}_{11} - \tilde{M}_{22} &= -\frac{1}{2}Z_t - \frac{1}{8}W_t,\\
    \tilde{N}_{22} - \tilde{N}_{11} &= \frac{1}{4}(Z_r - Z_s) 
    - \frac{1}{8} Z_r W_s - \frac{1}{8}(Z_s - 1)W_r - \frac{1}{8} W_s 
    - \frac{1}{16} W_r W_s.
  \end{align*}
  Applying the symmetric identity $(Z_s-1)W_r = (Z_s-1)(Z_r-1)C = (Z_r-1)W_s$, 
  the second difference simplifies to:
  \begin{align*}
    \tilde{N}_{22} - \tilde{N}_{11} &= \frac{1}{4}(Z_r - Z_s) 
    - \frac{1}{8} Z_r W_s - \frac{1}{8}(Z_r - 1)W_s - \frac{1}{8} W_s 
    - \frac{1}{16} W_r W_s\\
    &= \frac{1}{4}(Z_r - Z_s) - \frac{1}{4} Z_r W_s - \frac{1}{16} W_r W_s.
  \end{align*}
  Substituting these differences back into the expression for $c_{21}$ 
  yields the desired formula.
\end{proof}

By analyzing the conditions under which these explicit entries simultaneously 
vanish, we can completely determine when the non-scalar commutator vanishes.

\begin{proposition}\label{CMCNEXEZ}
  Assume $k \not\equiv 0, 4, 5 \pmod{6}$. Recalling that the variables 
  $Z_l = e^{2i l \theta} = e^{i \frac{l(1+k)}{6} 2\pi}$ (for $l = t, s, r$) are 
  roots of unity, the necessary and sufficient condition for the commutator 
  $[\tilde{M}_t,\tilde{N}_{r, s}]=0$ to vanish is that either $Z_r=Z_s=1$ 
  or $Z_t=Z_r=Z_s$ holds.
\end{proposition}

\begin{proof}
  From Proposition \ref{CMCNC}, if $[\tilde{M}_t,\tilde{N}_{r, s}]=0$, then 
  we must have $c_{11}=c_{22}=0$ and $c_{21}=0$. This gives us the following 
  system of equations:
  \begin{align}
    & 2Z_t(Z_r - Z_s) + W_s(Z_t - Z_r) = 0,\label{MNCC1}\\
    & 4Z_t(Z_r - Z_s) - Z_tW_rW_s + Z_rW_sW_t = 0.\label{MNCC2}
  \end{align}
  Let us assume, for the sake of contradiction, that $Z_t \neq Z_r$ and 
  $Z_s \neq 1$ (which implies $W_s \neq 0$). 
  From equation \eqref{MNCC1}, we isolate $4Z_t(Z_r - Z_s) = -2W_s(Z_t - Z_r)$. 
  Substituting this into equation \eqref{MNCC2} yields:
  \begin{equation*}
    -2W_s(Z_t - Z_r) - Z_tW_rW_s + Z_rW_sW_t = 0.
  \end{equation*}
  Dividing the entire equation by the non-zero quantity $W_s$, we obtain 
  $-2(Z_t - Z_r) - Z_tW_r + Z_rW_t = 0$. 
  Next, we substitute the definitions $W_r = (Z_r-1)C$ and $W_t = (Z_t-1)C$ 
  into this equation:
  \begin{align*}
    0 &= -2(Z_t - Z_r) - Z_t(Z_r - 1)C + Z_r(Z_t - 1)C \\
    &= (C - 2)(Z_t - Z_r).
  \end{align*}
  Since we assumed $Z_t \neq Z_r$, this forces $C=2$. However, under the 
  condition $k \not\equiv 0, 4 \pmod{12}$, the constant 
  $C = \csc((1+k)\pi/6)$ cannot be equal to 2, which presents a contradiction. 
  Therefore, our initial assumption must be false, meaning we must have 
  either $Z_s=1$ or $Z_t=Z_r$.

  Case 1: If $Z_s=1$, then $W_s=0$. Substituting this into equation 
  \eqref{MNCC1} yields $2Z_t(Z_r - 1) = 0$. Since $|Z_t|=1$, this forces 
  $Z_r=1$. Thus, $Z_r=Z_s=1$. In this case, it is straightforward to verify 
  that all entries in Proposition \ref{CMCNC} trivially vanish. Moreover, 
  under this condition, $\tilde{N}_{1,1}$ reduces to $\frac{1}{4}E$, which 
  evidently commutes with any $Z_t$.

  Case 2: If $Z_t=Z_r$, substituting this into equation \eqref{MNCC1} 
  yields $2Z_r(Z_r - Z_s) = 0$. Since $Z_r \neq 0$, this forces $Z_r=Z_s$. 
  Thus, $Z_t=Z_r=Z_s$. In this scenario as well, all entries trivially 
  cancel out to 0. In fact, when $Z_t=Z_r=Z_s$, the matrices $M(t)$ and 
  $N(t,t)$ are both powers of the exact same matrix $\rho(T^t S)$, 
  guaranteeing their commutativity.
\end{proof}

With these conditions rigorously established, the proof of our main theorem 
for this section is now complete.

\begin{proof}[Proof of Theorem \ref{THMNC}]
  The theorem is a direct consequence of synthesizing the explicit formulas 
  from Theorems \ref{THEXFM} and \ref{THEXFN}, the commutator factorization 
  in Proposition \ref{RWCMMNCMCN}, and the zero-commutator conditions 
  derived in Proposition \ref{CMCNEXEZ}.
\end{proof}

\section{Concluding Remarks}

In the process of classifying the modular bases, we completely determined the
commutativity of the matrices $M(t)$ and $N(r,s)$ generated by the monodromy
representations $\rho(T)$ and $\rho(S)$. It is worth noting that these monodromy 
matrices satisfy quadratic relations and the braid relation, which inherently 
equips the solution space with the structure of the Iwahori-Hecke 
algebra (or Temperley-Lieb algebra) \cite{GN, MMS}.

From an algebraic perspective, finding a family of commuting matrices with spectral
parameters, such as $N(r,s)N(u,v) = N(u,v)N(r,s)$ as classified in Theorem \ref{THNNC}
and $M(t)N(r,s) = N(r,s)M(t)$ as classified in Theorem \ref{THMNC}, 
is perfectly analogous to constructing a family of commuting transfer matrices 
in quantum integrable systems (e.g., the XXZ model associated with the quantum
group $U_q(\mathfrak{sl}_2)$).

While the primary goal of this paper is the number-theoretic classification
of $k$ for the Kaneko-Zagier equation \cite{KZ, KK}, the exact commutativity conditions
we derived might provide explicit realizations of such commuting operators. 
Exploring the direct connection between these modularity conditions 
and the Yang-Baxter equation would be a fascinating subject for future study.

\section*{Acknowledgements}
The authors would like to express their sincere gratitude to Professor Abdellah Sebbar
and his assistant for their careful reading of the manuscript. Their insightful comments
and invaluable suggestions, particularly in pointing out a gap in the proof of a proposition, 
have significantly improved the rigor and presentation of this paper.

\bibliographystyle{plain}
\bibliography{kzsda}
\nocite{*}

\end{document}